\documentclass[12pt]{amsart}

\usepackage{amsmath,amssymb,lscape}
\usepackage[all]{xy}

\newtheorem{theorem}{Theorem}[section]
\newtheorem{proposition}[theorem]{Proposition}
\newtheorem{lemma}[theorem]{Lemma}
\newtheorem{cor}[theorem]{Corollary}

\newcommand{\conj}{\overline}

\theoremstyle{definition}
\newtheorem{definition}[theorem]{Definition}
\newtheorem{example}[theorem]{Example}

\newtheorem{problem}[theorem]{Problem}

\theoremstyle{remark}
\newtheorem{remark}[theorem]{Remark}

\numberwithin{equation}{section}

\def\DJ{{\hbox{D\kern-.8em\raise.15ex\hbox{--}\kern.35em}}}
\def\DJo{$\;$\kern-.4em
    \hbox{D\kern-.8em\raise.15ex\hbox{--}\kern.35em koori\'c}}

\def\al{{\alpha}}
\def\be{{\beta}}

\def\la{{\lambda}}

\def\bR{{\mathbb {R}}}
\def\bC{{\mathbb {C}}}

\def\bH{{\mathbb {H}}}

\def\pA{{\mathcal A}}

\def\pF{{\mathcal F}}
\def\pI{{\mathcal I}}
\def\pJ{{\mathcal J}}
\def\pL{{\mathcal L}}

\def\pP{{\mathcal P}}
\def\pR{{\mathcal R}}
\def\pS{{\mathcal S}}

\def\pU{{\mathcal U}}
\def\pV{{\mathcal V}}
\def\pK{{\mathcal K}}
\def\pW{{\mathcal W}}
\def\pO{{\mathcal O}}

\def\tr{{\rm tr\;}}
\def\Tr{{\rm Tr\;}}
\def\End{{\mbox{\rm End}}}

\def\diag{{\mbox{\rm diag}}}
\def\GL{{\mbox{\rm GL}}}

\def\Ort{{\mbox{\rm O}}}
\def\Sp{{\mbox{\rm Sp}}}
\def\U{{\mbox{\rm U}}}

\def\Dx{\frac{\partial}{\partial x}}
\def\Dy{\frac{\partial}{\partial y}}

\renewcommand{\subjclassname}{\textup{2000} Mathematics Subject
Classification}

\begin{document}

\title[Quaternionic matrices]
{Quaternionic matrices: Unitary similarity, simultaneous
triangularization and some trace identities}

\author[D.\v{Z}. \DJ okovi\'{c} and B. Smith]
{Dragomir \v{Z}. \DJ okovi\'{c} and Benjamin H. Smith}

\address{Department of Pure Mathematics, University of Waterloo,
Waterloo, Ontario, N2L 3G1, Canada}

\email{djokovic@uwaterloo.ca} \email{bh2smith@uwaterloo.ca}

\thanks{
The first author was supported by the NSERC Discovery Grant A-5285,
and the second by an NSERC Undergraduate Student Research Award}


\date{}

\begin{abstract}
    We construct six unitary trace invariants for $2\times 2$ quaternionic
    matrices which separate the unitary similarity classes of such
    matrices, and show that this set is minimal. We have discovered a
    curious trace identity for two unit-speed one-parameter subgroups of
    $\Sp(1).$ A modification gives an infinite family of trace identities
    for quaternions as well as for $2\times 2$ complex matrices.
    We were not able to locate these identities in the literature.

    We prove two quaternionic versions
    of a well known characterization of triangularizable subalgebras
    of matrix algebras over an algebraically closed field. Finally
    we consider the problem of describing the semi-algebraic set of
    pairs $(X,Y)$ of quaternionic $n\times n$ matrices which are
    simultaneously triangularizable. Even the case $n=2,$ which we
    analyze in more detail, remains unsolved.
\end{abstract}
\maketitle \subjclassname { 15A33,15A18,16R30}

\section{Introduction}
We denote by $\bR,\bC,\bH$ the field of real numbers, complex
numbers and the division ring of real quaternions, respectively.
Throughout we use $D$ to represent an element from the set
$\{\bR,\bC,\bH\}$ and $M_n(D)$ denotes the algebra of $n\times n$
matrices over $D$. Also, we let $\pU_n(D)$ resp. $\pL_n(D)$ denote
the space of upper resp. lower triangular matrices in $M_n(D)$. In
the case where $D=\bH$ we will omit parentheses and write
$M_n,\pU_n,\pL_n$.

The maximal compact subgroup of the general linear group $\GL_n(D)$,
$\Ort(n)$ in the real case, $\U(n)$ in the complex case and $\Sp(n)$
in the quaternionic case, acts on $M_n(D)$ by conjugation, i.e.,
$(X,A)\mapsto XAX^{-1},\,A\in M_n(D).$ It is an old problem to
determine an explicit minimal set of generators for the algebra of
polynomial invariants for this action. Explicit results are known
(in all three cases) for small values of $n.$ Let us mention that
the algebra of real polynomial $\Sp(n)$-invariants is generated by
the trace functions $\Tr \left(w(X,X^*)\right)$ where $\Tr:M_n\rightarrow\bR$ is
the quaternionic trace (see next section for definition) and $w$ is
any word in two letters.

The first question that we consider is that of separating the orbits
of the above action by means of a minimal set of polynomial
invariants. The real and complex cases have been studied
extensively. For instance, Pearcy shows in \cite{Pe} that $A,B\in
M_2(\bC)$ are unitarily similar if and only if $\tr \left(X\right)$, $\tr\left( X^2\right) $
and $\tr \left(XX^*\right)$ take the same values on $A$ and $B$. He also gives a
list of nine words in $X$ and $X^*$ whose traces distinguish the
unitary orbits in $M_3(\bC)$. This is reduced to a minimal set of
seven words by Sibirski\u\i\, \cite{KS}. As far as we know, there
are no such results in the quaternionic case except for the case
$M_1$, which is trivial.

In section 3 we extend the first of Pearcy's results to $M_2$, i.e.,
the $2\times 2$ quaternionic matrices. In this case we show that six
words suffice (see Theorem \ref{unitsim2}), and that our set of six
words is minimal

In section 4 we consider two unit-speed one-parameter subgroups of
$\Sp(1),$ say
\[ \phi_p(s)=e^{ps},\quad \phi_q(t)=e^{qt}\]
where $p$ and $q$ are pure quaternions of norm 1. We prove (Theorem
\ref{uident}) that the real part of
\[ \prod_{i=1}^k\phi_p(s_i)\phi_q(t_i);\quad s_i,t_i\in\bR\]
remains the same when $p$ and $q$ are switched.

This fact remains true for arbitrary pure quaternions $p$ and $q$
provided we take $k=2$ and set $s_1=t_1$ and $s_2=t_2.$ 
{}From here we obtain an infinite family of trace identities for quaternions
(see Proposition \ref{uident2} and it's corollaries), which we were not
able to find anywhere in literature. For instance we show that
\[  \Tr \left(x^my^mx^ny^n\right) = \Tr \left(y^mx^my^nx^n\right)  \]
is valid for all quaternions $x,y$ and nonnegative integers $m,n$.
One can easilly convert these identities into trace identities for
$M_2(\bC).$

Section 5 is about the simultaneous triangularization of subalgebras
of $M_n$. In \cite{RR} Radjavi and Rosenthal give several
characterizations of triangularizability of unital subalgebras of
finite dimensional linear operators over an algebraically closed
field. By changing the equality in part (iv) of \cite[Theorem
1.5.4]{RR} to an inequality we are able to extend that result to the
quaternionic case (see Theorem \ref{nilpotent}). Next, we observe a
peculiar polynomial equation which is satisfied on any
triangularizable subalgebra $\pA\subseteq M_n$, namely that $\Tr
\left([X,Y]^3\right)=0$, and we investigate its significance. We introduce the
concept of quasi-triangularization (generalizing the
triangularization) which refers to the possibility of obtaining a
simultaneous block upper triangular form with the diagonal blocks
restricted to $M_1$ or $M_2(\bC)$. Based on our trace identity for
complex matrices given in section 4, we find that a unital
subalgebra $\pA\subseteq M_n$ is quasi-triangularizable if and only
if the identity $\Tr\left([X,Y]^3\right)=0$ is valid on $\pA$.

In section 6 we explore the semi-algebraic set $\pW_n$ of pairs of
quaternionic matrices which are simultaneously triangularizable.
Hence $\pW_n$ is generated from $\pU_n\times\pU_n$ via the
simultaneous conjugation action of the group $\GL_n(\bH).$ One can
replace here $\GL_n(\bH)$ by $\Sp(n)$ and deduce that $\pW_n$ is
closed (in the Euclidean topology). For generic $A\in M_n$, i.e.,
one with $n$ distinct eigenvalues, we find that the fibre of the
first projection $\pW_n\rightarrow M_n$ is the union of $n!$ real
vector spaces, each of dimension $2n(n+1),$ with a pairwise
intersection a common subspace of dimension $\geq4n$. We deduce that
the dimension of $\pW_n$ is $2n(3n+1)$. We also construct two
infinite families of polynomial equations (and some inequalities)
which are satisfied on $\pW_n$.

The problem of pairwise triangularizability in $M_2$ is of special
interest as the first nontrivial case of the above mentioned general
problem. Here, the set $\pW_2$ can be defined geometrically as the
set of quaternionic matrix pairs which share a common eigenvector.
In \cite{F} Friedland describes exactly when this occurs in the
complex case, see Theorem \ref{Friedland} below. In section 7, we
look at his result and attempt to extend it to the quaternionic
case. We now give some details.

Let $\pP_2$ be the algebra of real polynomial functions on
$M_2\times M_2,$ and $\pP_2'$ resp. $\pP_2''$ the subalgebra of
$\GL_2(\bH)$ resp. $\Sp(2)$-invariants. Let $\pI_2\subseteq\pP_2$ be
the ideal of functions that vanish on $\pW_2$, and set
$\pI_2'=\pI_2\cap\pP_2'$ and $\pI_2''=\pI_2\cap\pP_2''$. The Zariski
closure $\overline{\pW}_2$ of $\pW_2$ is the set of common zeros of
$\pI_2$, and the same is true for the ideal $\pI_2''$ of $\pP_2''$.
While the codimension of $\pW_2$ in $M_2\times M_2$ is only 4, we
can show that a minimal set of generators of $\pI_2'$ has cardinal
$\geq 92.$ Let $\pJ'_m\subseteq\pI_2'\subseteq\pP_2'$ be the ideal
of $\pP_2'$ generated by all polynomials $f\in\pI_2'$ of (total)
degree $\leq m.$ We have constructed a minimal set of generators of
$\pJ'_m$ for $m\leq 14$ (see Table 2 for the generators of
$\pJ'_9$).

In the last section, 8, we state four open problems.

We thank the referee and R. Guralnick for their valuable comments and suggestions.

\section{Preliminaries}

Let $\bH=\{a+\mathsf{i}b+\mathsf{j}c+\mathsf{k}d:a,b,c,d\in\bR\}$
represent the skew-field of real quaternions. We shall identify the
field $\bC$ with the subfield $\{a+\mathsf{i}b:a,b\in\bR\}$ of $\bH$.
For a quaternion $q=a+\mathsf{i}b+\mathsf{j}c+\mathsf{k}d$ we define
the \emph{norm}, \emph{real part}, \emph{pure part} and 
\emph{conjugate}\ of $q$ in the usual fashion as:
\begin{eqnarray*}
&& |q|^2:=a^2+b^2+c^2+d^2, \\
&& \Re(q):=a, \\
&& \mathsf{i}b+\mathsf{j}c+\mathsf{k}d \text{ and} \\
&& \conj{q}:=a-\mathsf{i}b-\mathsf{j}c-\mathsf{k}d,
\end{eqnarray*}
respectively. The \emph{adjoint} of a matrix $A\in M_n$ is given by
$A^*=\overline{A}^T$, which is also known as the conjugate-transpose
of $A$. With this, we can define the (compact) \emph{symplectic
group}, $\Sp(n)$, as the collection of quaternionic unitary
matrices, namely
$$\Sp(n):=\{U\in M_n:U^*U=I_n\},$$
where $I_n$ is the identity matrix see e.g. \cite{CC}. Consider the
equivalence relation $\sim$ defined on $M_n$ by the conjugation
action of $\Sp(n)$. To be precise, we have:
\[ A\sim B \leftrightarrow \exists U\in\Sp(n),\quad A=UBU^{-1}. \]
This shall be referred to as \emph{$\Sp(n)$-equivalence}. We will
speak of individual $\Sp(n)$-equivalence class for a given matrix
$A\in M_n$ and thus, denote this orbit by
$\pO_A=\{UAU^{-1}:U\in\Sp(n)\}$. It is well known that $M_n$ can be
embedded nicely into $M_{2n}(\bC)$. For this purpose, given $A\in
M_n$ we write $A=A_1+\mathsf{j}A_2$ with $A_1,A_2\in M_n(\bC).$ Then
\[ \chi_n:M_{n}\rightarrow M_{2n}(\bC),\quad
    \chi_n(A)=\begin{bmatrix} A_{1} & -\overline{A}_2\\
                              A_{2} &  \overline{A}_1\end{bmatrix}, \]
is an injective homomorphism of $\bR$-algebras. 
{}From this, the quaternionic matrices inherit various analogous properties
regarding invertibility, triangularizability, canonical forms, decomposition,
determinants, numerical range and more. See \cite{FZ} for detailed
results on quaternionic linear algebra.

Given $A\in M_n$, there exists $P\in\GL_n(\bH)$ such that
$PAP^{-1}\in \pU_n$ with successive diagonal entries
$\la_1,\dots,\la_n\in\bC$ and imaginary parts $\Im(\la_i)\geq0$. The
sequence $(\la_1,\dots,\la_n)$ is unique up to permutation and we
refer to $\la_1,\dots,\la_n$ as the \emph{eigenvalues} of $A$. We
shall use the word ``eigenvalues", in the context of quaternionic
matrices, exclusively in this sense. Note that the eigenvalues of
the complex matrix $\chi_n(A)$ are $\la_1,\dots,\la_n$ and
$\bar{\la}_1,\dots,\bar{\la}_n$ (counting multiplicities).

We seek to classify exactly when two $2\times2$ quaternionic
matrices are $\Sp(2)$-equivalent using polynomial functions which
remain constant on the equivalence classes. It is known that the
algebra of polynomial invariants for complex matrices under the
action of conjugation by unitary group is generated by a finite
number of particular \emph{known} trace functions on $M_n(\bC)$ for
$n=2,3,4$. Since $M_2$ embeds  into $M_4(\bC)$ as seen by $\chi_2$,
it is only natural to assume a classification of this type can be
extended in some way to the quaternionic case. That is, we should be
capable of defining explicitly which functions separate orbits.

First, we will need to introduce the notion of trace for quaternions
and matrices of such which will be preserved by $\chi_n$ above.
For the general definition of the reduced trace and the reduced
norm for central simple algebras we refer the reader to
\cite{RP,GS}.

\begin{definition}
The \emph{quaternionic trace} of $A=[a_{ij}]\in M_{n}$ is,
\[ \Tr(A):=\sum_{i=1}^n(a_{ii}+\bar{a}_{ii})=
2\Re\left( \sum_{i=1}^na_{ii}\right). \] In particular, when $n=1$
we have $$\Tr(q)=q+\bar{q}=2\Re(q)$$
\end{definition}

It is important to notice the clear distinction between the usual
trace, $\tr,$ on a matrix algebra over a field and the quaternionic
analog presented above. For instance, we have $\tr I_n = n$ while
$\Tr I_n = 2n$. Some properties which follow directly from our
definition as well as the simple fact that $\Re(pq)=\Re(qp)$ for
all quaternions $p,q$ are as follows:\\

Let $A,B\in M_n$, $P\in \GL_{n}(\bH)$, $U\in\Sp(n)$ and $w$ be any
word on two letters, then
\begin{enumerate}
    \item $\Tr (A) =\tr \chi_n(A)$
    \item $\Tr(A+B)=\Tr (A)+\Tr (B) $
    \item $\Tr (\la A) =\la \Tr (A) ,\quad \la\in\bR$
    \item $\Tr (AB) =\Tr (BA) $
    \item $\Tr \left(w(PAP^{-1},PBP^{-1})\right)=\Tr \left(w(A,B)\right)$
    \item $\Tr \left(w(UAU^*,UA^*U^*)\right)=\Tr \left(w(A,A^*)\right)$
\end{enumerate}

Observe, properties (2),(3) and (6) tell us that the quaternionic
trace of any $\bR$-linear combination of any words on $\{A,A^*\}$ is
invariant under the action of $\Sp(n)$.

\section{$\Sp(2)$-equivalence of matrices}
For our purposes it is essential to introduce the following six
trace polynomials on $M_2$
\[ (p_{1},p_{2},p_{3},p_{4},p_{5},p_{6}):=
    \Tr(A,A^2,A^3,A^4,AA^*,A^2 A^{*2}).\]
We proceed in showing that these form a minimal set of
$\Sp(2)$-invariants that separate orbits in $M_2$.

 Let us first describe a simple canonical form for
$2\times2$ quaternionic matrices under $\Sp(2)$-equivalence.

\begin{definition}\label{K}
Denote by $\pK$ the set of all matrices of the form
\[ A=\begin{bmatrix}\al & z\\0 & \be\end{bmatrix} \]
such that $\al,\be\in\mathbb{C}$; $\Im(\al),\Im(\be)\geqslant 0$;
$z=z_{1}+\mathsf{j}z_{3}$; $z_{1},z_{3}\geqslant 0$; and if $\al$,
or $\be$ is real then $z_3=0$.
\end{definition}

Notice that $\pK$ is a semi-algebraic set of dimension 6.
Eventually, we will see $\pK$ intersects each $\Sp(2)$-equivalence
class at either one or two points. The following result shows  that
$\pK$ meets every orbit $\pO_A$ at least once.

\begin{lemma}\label{uptri}
Every $2\times2$ quaternionic matrix is $\Sp(2)$-equivalent to some
matrix $A\in\pK$.
\end{lemma}

\begin{proof}
First, by the generalization of Schur's theorem for quaternionic
matrices found in \cite{FZ} any matrix in $M_2$ is
$\Sp(2)$-equivalent to a matrix of the form
\[ \begin{bmatrix}\al & z\\0 & \be\end{bmatrix}, \]
with $\al,\be\in\bC$; and $\Im(\al),\Im(\be)\geqslant 0$. If we
write $z=c_{1}+\mathsf{j}c_{2}\text{; }c_{1},c_{2}\in\bC$, we can
reduce our matrix as follows,
    \begin{eqnarray*}
        \begin{bmatrix}u & 0\\0 & v\end{bmatrix}
        \begin{bmatrix}\al & c_{1}+\mathsf{j}c_{2}\\0 & \be\end{bmatrix}
        \begin{bmatrix}u^{-1}& 0\\0 & v^{-1}\end{bmatrix}
        &=&
        \begin{bmatrix}\al & |c_{1}| + \mathsf{j}|c_{2}|\\0 &
        \be\end{bmatrix}.
    \end{eqnarray*}
    It suffices to choose unit complex numbers $u$ and $v$ such that
    $uc_{1}v^{-1}=|c_{1}|$ and
    $\bar{u}c_{2}v^{-1}=|c_{2}|,$
    which is always possible. If, say, $\be$ is real then we can set
    $u=1$ and choose a unit quaternion $v$ such that $zv^{-1}$ is
    real nonnegative.
\end{proof}

Now we provide a technical result which simplifies the form of $p_6$
when restricted to $\pK$. This will be useful for later
computations.

\begin{lemma}\label{inv}
    Let $A=\begin{bmatrix}\al & z\\0 & \be\end{bmatrix}$ be as
    in Definition \ref{K}. In particular $z=z_1+\mathsf{j}z_3$ with
    $z_1,z_3\geqslant 0$. Then,
    \[ \frac{1}{2}p_{6}(A)=|\al|^4 + |\be|^4 + |\al+\bar{\be}|^2|z|^2 +
        z_{1}^2(\al-\bar\al)(\bar\be-\be). \]
\end{lemma}

\begin{proof} We have
    \[ p_{6}(A)=\Tr(A^2 A^{*2})=2\left(|\al|^4 +|\be|^4 +
        (|\al|^2 + |\be|^2 )|z|^2 +2\Re(\al z\bar{\be}\bar{z})\right) \]
    and
    \[ z\bar{\be} \bar{z}=(\bar{\be}z_{1}^2 + \be z_{3}^2) +
        \mathsf{j}z_1 z_3 (\be + \bar{\be}). \]
    Notice that the second term
    above is a pure quaternion even upon multiplication by $\al$ and so
    has zero real part. So
    \begin{align*}
        2\Re(\al z\bar{\be}\bar{z})&=z_1^2(\al\bar{\be}+\bar{\al}\be)+z_3^2(\al\be+\bar{\al}\bar{\be})&\\
        &=|z|^2(\al\be+\bar{\al}\bar{\be})+z_1^2(\al\bar{\be}+\bar{\al}\be-\al\be+\bar{\al}\bar{\be})&\\
        &=|z|^2(\al\be+\bar{\al}\bar{\be})+z_1^2(\al-\bar{\al})(\bar{\be}-\be).
    \end{align*}
\end{proof}

With this formulation in place, we can classify $\Sp(2)$-equivalence
between matrices which lie in $\pK$. In fact we shall prove that an
orbit $\pO_A$ meets $\pK$ in two points when $A$ has distinct
eigenvalues and in a single point otherwise.

\begin{theorem}\label{unitsim}
    If $A=\begin{bmatrix}\al & z\\0 & \be\end{bmatrix}$ and
    $B=\begin{bmatrix}\gamma & w\\0 & \delta\end{bmatrix}$ belong to
    $\pK,$ then
    \[ A\sim B \iff  z=w\ \&\ \{\al,\be\}=\{\gamma,\delta\}. \]
\end{theorem}

\begin{proof}
    When $A\sim B$ we get from $p_k(A)=p_k(B)$, $k\in\{1,2,3,4\}$, that
    the sets of eigenvalues, $\{\al,\be,\bar{\al},\bar{\be}\} \text{
    and }\{\gamma,\delta,\bar{\gamma},\bar{\delta}\}$, for
    $\chi_2(A),\chi_2(B)$ are the same. So we get that
    $\{\al,\be\}=\{\gamma,\delta\}$ as well. Also, from $p_5(A)=p_5(B)$
    which is $|\al|^2+|\be|^2+|z|^2=|\gamma|^2+|\delta|^2+|w|^2$, we can
    see that $|z|^2=|w|^2$.
    Recall that $z=z_1+\mathsf{j}z_3$ and $w=w_1+\mathsf{j}w_3$, where
    $z_1,z_3,w_1,w_3$ are real and nonnegative. Now, $p_6(A)=p_6(B)$ and
    Lemma \ref{inv} above show that $z_1^2=w_1^2$, and since both $z_1$
    and $w_1$ are nonnegative we have $z_1=w_1$. It follows that also $z=w$.

    To prove the converse, we may assume that $A\neq B$. Thus
    $\delta=\al,\gamma=\be$ and $\al\neq \be$. So, we know $A$ is
    $\Sp(2)$-equivalent to a matrix $A'=\begin{bmatrix}\be &
    w'\\0 & \al\end{bmatrix}\in\pK$, since Schur's theorem allows us to
    place the eigenvalues in any order along the diagonal. Hence, by
    the first part of the proof and the hypothesis we have $w'=z=w$ and so $A\sim A'=B$.
\end{proof}

With Lemma \ref{uptri} along with Theorem \ref{unitsim}, we have
reached the promised canonical form for $2\times2$ quaternionic
matrices under $\Sp(2)$-equivalence. It is unique up to permutation
of the diagonal entries.

Now we may begin looking to find polynomial invariants which
separate the $\Sp(2)$-equivalence classes. Also, it is ideal for
computational purposes to obtain the least number of these
polynomials which do the job.

\begin{theorem}\label{unitsim2}
    Two matrices $A,B \in M_{2}$ are $\Sp(2)$-equivalent if
    and only if the following six equations hold:
    \begin{eqnarray*}
    && \Tr \left(A^{i}\right)=\Tr \left(B^{i}\right),\,\ i \in\{1,2,3,4\};\\
    && \Tr \left(AA^{*}\right)=\Tr \left(BB^{*}\right);\\
    && \Tr \left(A^{2}A^{*^{2}}\right) =\Tr \left(B^{2}B^{*^{2}}\right),
    \end{eqnarray*}
    i.e., $p_k(A)=p_k(B)$ for $1\leq k\leq 6.$ Moreover, this is a
    minimal set of invariants with the mentioned property.
\end{theorem}

\begin{proof}
    First, if $A\sim B$, our trace equations are trivially satisfied as
    $ \Tr w(A,A^*) =\Tr w(B,B^*)$ for all words $w$ on two letters.

    Conversely, suppose the given set of traces match for $A$ and $B$.
    By Lemma \ref{uptri} we may assume that
    \[ A=\begin{bmatrix}\al&z\\0&\be\end{bmatrix},\,
        B=\begin{bmatrix}\gamma&w\\0&\delta\end{bmatrix}\in\pK. \]
    We know the first four invariants of $A,B$ uniquely determine the
    sets $\{\al,\be\}$, $\{\gamma,\delta\}$ respectively and since these
    invariants are the same, we have that these sets are equal. It
    remains to show that $z=w$ which can be done using $p_5$ and $p_6$.
    We have from $p_5(A)=p_5(B)$ that
    \[ |\al|^2+|\be|^2+|z|^2=|\gamma|^2+|\delta|^2+|w|^2. \]
    As $\{\al,\be\}=\{\gamma,\delta\}$, we have $|z|=|w|$. Similarly,
    from $p_6(A)=p_6(B)$ and Lemma \ref{inv} we get
    \begin{eqnarray*}
        && |\al|^4+|\be|^4+|\al+\bar{\be}|^2|z|^2+z_{1}^2(\al-\bar\al)(\bar\be-\be) \\
        && \quad =|\gamma|^4+|\delta|^4+|\gamma+\bar{\delta}|^2|z|^2+
        w_{1}^2(\gamma-\bar\gamma)(\bar\delta-\delta),
    \end{eqnarray*}
    and so $z_1^2=w_1^2$.
    {}From our description of $\pK$ we know $z_1,w_1\geq0$. 
    This implies that $z_1=w_1$ and thus, $z=w$. Now, it
    follows from Theorem \ref{unitsim} that $A\sim B$.

    Finally, to prove minimality, we provide pairs of matrices, each of
    which agree on all but one invariant from our list. In Table 1 the
    matrices in the $k$-th row have distinct values of $p_k$ only.
    \begin{center}
    \textbf{Table 1:} Examples for minimality
    $$
     \begin{array}{c c c}
        1.  & \left[\begin{array}{ c c }
               \sqrt{3}-\mathsf{i} & 0 \\0 &-\sqrt{3}+\mathsf{i}\end{array}\right]
            & \left[\begin{array}{ c c }
                -\sqrt{3}+\mathsf{i}&0\\0&-\sqrt{3}-\mathsf{i}\end{array}\right]\\ \\
        2. & \left[\begin{array}{ c c }
                2+\mathsf{i}&1\\0&-2-\mathsf{i}\end{array}\right]
            & \left[\begin{array}{ c c }
                    1+2\mathsf{i}&-1\\0&-1-2\mathsf{i}\end{array}\right]\\ \\
        3. & \left[\begin{array}{ c c }
                -1+2\mathsf{i} & 0 \\0 &1\end{array}\right]
            & \left[\begin{array}{ c c }
                    -1&0\\0&1+2\mathsf{i}\end{array}\right]\\ \\
        4. & \left[\begin{array}{ c c }
                0&\sqrt{6}+\mathsf{j}\\0&\sqrt{2}\mathsf{i}\end{array}\right]
            &\left[\begin{array}{ c c }
            \mathsf{i}&\sqrt{3}+2\mathsf{j}\\0&-\mathsf{i}\end{array}\right]\\ \\
        5. & \left[\begin{array}{ c c }
                0&0\\0&0\end{array}\right]
            & \left[\begin{array}{ c c }
                   0&1\\0&0\end{array}\right]\\ \\
        6. & \left[\begin{array}{ c c }
                \mathsf{i}&1\\0&\mathsf{i}\end{array}\right]
            & \left[\begin{array}{ c c }
                \mathsf{i}&\mathsf{j}\\0&\mathsf{i}\end{array}\right]\\
     \end{array}
    $$
    \end{center}
    Thus, we have shown that our set of invariants is minimal.
\end{proof}

Let us point out that the complex analog of this theorem is not valid.
The reason is that, in the complex case, the polynomial invariants
do not separate the orbits.

\section{Some trace identities for quaternions}

It is well known that every one parameter subgroup of $\Sp(1)$ has
the form
$$
\phi_p(s)=e^{sp}:=\sum_{i\geq0}\frac{1}{i!}s^ip^i, \quad s\in\bR,
$$
for a unique pure quaternion $p$. Note that $\phi_{\la
p}(s)=\phi_p(\la s)$ for all real $\la$ and
\begin{equation}\label{conj}
    u\phi_p(s)u^{-1}=\phi_{upu^{-1}}(s)
\end{equation}
for any nonzero quaternion $u$.

\begin{theorem}\label{uident}
    If $p$ and $q$ are pure unit quaternions, then
    \begin{equation}\label{onepsg}
         \Tr\left(\prod_{i=1}^k\phi_p(s_i)\phi_q(t_i)\right)=\Tr\left(\prod_{i=1}^k\phi_q(s_i)\phi_p(t_i)\right).
    \end{equation}
    is valid for any real numbers $s_1,\dots,s_k$ and $t_1,\dots,t_k$.
\end{theorem}

\begin{proof}
    Since $|p|=|q|=1$ there exists a $180^\circ$-rotation, in the
    3-dimen\-si\-on\-al space of pure quaternions, which interchanges $p$
    and $q$. Consequently, there exists a pure unit quaternion $u$
    such that $upu^{-1}=q$ and $uqu^{-1}=p$ (if $p+q\neq 0$ one can
    choose $u$ in the direction of $p+q$, and otherwise just to be
    orthogonal to $p$.) It suffices now to observe that conjugation
    by $u$ interchanges the two products in (\ref{onepsg}).
\end{proof}

The identity (\ref{onepsg}) is not valid for arbitrary pure
quaternions $p$ and $q$. Let us look at the special cases where
\begin{equation}\label{*}
    s_1=t_1,s_2=t_2,\dots,s_k=t_k.
\end{equation}

\begin{proposition}\label{uident2}
    Let $p$ and $q$ be arbitrary pure quaternions, then
    \begin{itemize}
      \item[(a)] For every $s,t\in\bR$, we have
        $$\Tr (\phi_p(s)\phi_q(s)\phi_p(t)\phi_q(t))=
                    \Tr (\phi_q(s)\phi_p(s)\phi_q(t)\phi_p(t));$$ 
	\item[(b)] When at least two of $r,s,t\in\bR$ are equal, we have
      \begin{eqnarray*}
		&& \Tr (\phi_p(r)\phi_q(r)\phi_p(s)\phi_q(s)\phi_p(t)\phi_q(t)) \\
            && \quad =\Tr (\phi_q(r)\phi_p(r)\phi_q(s)\phi_p(s)\phi_q(t)\phi_p(t)).
	\end{eqnarray*}
                    
    \end{itemize}
\end{proposition}
\begin{proof}
    Let us prove (a). We can write $p=\la p_0$ and $q=\mu q_0,$
where $p_0$ and $q_0$ are pure unit quaternions and $\la,\mu\geq0.$
{}From (\ref{conj}) we may assume that $p_0=\mathsf{i}$ and
$q_0=\mathsf{i}\cos\rho+\mathsf{j}\sin\rho.$ Then we have
\[ \phi_p(s)=\phi_{\la p_0}(s)=\phi_{p_0}(\la s)=\cos\la s + \mathsf{i}\sin\la s, \]
\[ \phi_q(s)=\cos\mu s + (\mathsf{i}\cos\rho + \mathsf{j}\sin\rho)\sin\mu s. \]
Next we get that
\begin{align*}
    \phi_p(s)\phi_q(s)=&\cos \la s \cos \mu s  -\sin \la s \sin \mu s \cos \rho &\\
    &+ \textsf{i}(\cos \la s \sin \mu s \cos \rho + \sin \la s \cos \mu s) &\\
    &+ \textsf{j}\cos \la s \sin \mu s \sin \rho &\\
    &+ \textsf{k}\sin \la s \sin \mu s \sin \rho .&
\end{align*}
{}From here, we compute the product
$\phi_p(s_1)\phi_q(s_1)\phi_p(s_2)\phi_q(s_2)$ and verify (we did it
using Maple) that its real part remains the same when $p$ and $q$
are switched. Hence, proving part (a).

The proof of the part (b) is similar and we omit the details.
\end{proof}

A discrete version of the last proposition can be extended to
arbitrary quaternions $x,y$ by considering their polar
decompositions. Thus, we have the following corollary.

\begin{cor}\label{qident}
    Let $m,n,r$ be nonnegative integers. Then
    \[ \Tr \left(x^my^mx^ny^n\right) = \Tr \left(y^mx^my^nx^n\right) \]
    is valid for all $x,y\in\bH.$ If $m,n,r$ are not distinct, then
    \[ \Tr \left(x^my^mx^ny^nx^ry^r\right) = \Tr \left(y^mx^my^nx^ny^rx^r\right) \]
    is also valid.
\end{cor}

As a result of this identity for quaternions we obtain another
important corollary about $M_{2}(\mathbb{C})$. Take note the trace
below is the standard trace, as denoted by the lowercase tr.

\begin{cor}\label{disc2}
    Let $m,n,r$ be nonnegative integers. Then
    \[ \tr \left(x^my^mx^ny^n\right) = \tr \left(y^mx^my^nx^n\right) \]
    is valid for all $x,y\in M_2(\bC).$ If $m,n,r$ are not distinct , then
    \[ \tr \left(x^my^mx^ny^nx^ry^r\right) = \tr \left(y^mx^my^nx^ny^rx^r\right) \]
    is also valid.
\end{cor}

\begin{proof}
We have a direct decomposition
$M_2(\bC)=\chi_1(\bH)\oplus\mathsf{i}\chi_1(\bH)$. Since the left
hand side is a complex analytic polynomial (in 8 indeterminates)
which we know from Corollary \ref{qident} vanishes on $\chi_1(\bH)$,
it follows that this polynomial must be identically zero and our
identity holds.
\end{proof}

The referee supplied a simple proof of these two corollaries
for any algebra with trace. For simplicity we sketch his argument
in the setting of Corollary \ref{qident}.
By expanding products, one can easily check that the identity
\[ \Tr((a+bx)(c+dy)(e+fx)(g+hy))=
	\Tr((c+dy)(a+bx)(g+hy)(e+fx)) \]
holds for all real scalars $a,b,c,d,e,f,g,h$ and quaternions
$x,y$. This implies the first identity since $x^m=u+vx$ etc
for some real scalars $u,v$. The second identity can be deduced
from the first by using the fact that the elements of $\bH$ are
quadratic over $\bR$.

For convenience we shall use the standard notation for the commutator.
That is $[A,B]:=AB-BA.$ Notice that
\begin{eqnarray}\label{jed}
    \Tr\left([A,B]^3\right) &=& 3\Tr(A^2B^2AB-B^2A^2BA) \\
	\nonumber &=& -3\Tr \left(AB^2A[A,B]\right),
\end{eqnarray}
as this will be useful for further results.

\section{Triangularizable subalgebras of $M_n$}

A set $\pS\subseteq M_n(D)$ is \emph{triangularizable} if there is a
$P\in \GL_n(D)$ such that $P\pS P^{-1}\subseteq \pU_n(D)$. In the
recent book \cite{RR} one can find several characterizations of
triangularizable complex subalgebras of $M_n(\bC)$. Some of these
results can be easily transferred to the quaternionic case while
others are no longer valid. In particular, it is trivial to see that
the next proposition does not hold for the quaternionic case. First,
we introduce the notion of permutable functions.

\begin{definition}\label{ptrace}
    Let $\pS\subseteq M_n(D)$ be a collection of matrices. For a
    function $f$ on $M_n(D)$, taking values in the center of $D$,
    we say $f$ is \emph{permutable on} $\pS$ if
    $$
        f(A_1A_2\dots A_k)=
        f(A_{\sigma(1)}A_{\sigma(2)}\dots A_{\sigma(k)})
    $$
    for all $A_1,A_2,\dots ,A_k\in\pS$ and all permutations
    $\sigma$ of $\{1,2,\dots,k\}$.
\end{definition}

One can find in \cite{HR} the following characterization of
triangularizable subalgebras of $M_n(\bC)$. See \cite{GG} for
generalization to other fields.

\begin{proposition}\label{permutable trace}
    Let $\pA\subseteq M_n(\bC)$ be a unital complex subalgebra. Then $\pA$
    is triangularizable if and only if trace is permutable on $\pA$.
\end{proposition}

The analogous assertion for (real) subalgebras of $M_n$ is invalid
because $\Tr$ is not permutable on $\pU_n$. This is due to the lack
of commutativity of $\bH$.

To prove our result giving a characterization of
triangu\-la\-ri\-za\-ble subalgebras of $M_n$ we shall use the
concept of quaternionic representations of real algebras.

\begin{definition} \label{quatrep}
    Let $\pA$ be an associative unital $\bR$-algebra. A
    \emph{quaternionic representation} of $\pA$ is a $\bR$-algebra
    homomorphism
    \[ \rho:\pA\rightarrow \End_{\bH}(\pV), \]
    where $\pV$ is a right quaternionic vector space. We also say that
    $\pV$ is a \emph{quaternionic (left) $\pA$-module}. We say that
    $\rho$ is \emph{irreducible} if $\pV$ is non-zero and has no proper
    non-zero $\pA$-invariant quaternionic subspaces.
\end{definition}

Let us briefly outline the basic facts regarding quaternionic
representations of finite-dimensional unital $\bR$-algebras $\pA$.
If $\pR$ is the radical of $\pA$, then
$\pA/\pR\cong\pA_1\times\dots\times\pA_s$, where each $\pA_i$ is a
simple algebra. Thus, each $\pA_k$ is isomorphic to one of the
algebras $M_r(\bR)$, $M_r(\bC)$ or $M_r(\bH)$, for some integer
$r\ge1$. In each of these three cases, the right quaternionic space
$\bH^r$ (column vectors) is an irreducible quaternionic
$\pA_k$-module, and also an irreducible quaternionic $\pA$-module.
In this way we obtain all irreducible quaternionic $\pA$-modules (up
to isomorphism). Moreover, the $\pA$-modules arising for different
values of $k$ are pairwise non-isomorphic.

\begin{remark}
    Let us also mention another useful fact: There exists a
subalgebra $\mathcal{B}\subseteq\pA$ such that
$\pA=\pR\oplus\mathcal{B}$ see \cite[Wedderburn--Malcev Theorem,
p.209]{RP}.
\end{remark}

The following proposition plays a crucial role in the sequel.

\begin{proposition} \label{subalgebra}
    Let $\pA\subseteq M_n$ be a unital subalgebra and
    $\rho:\pA\rightarrow M_r$ an irreducible quaternionic
    representation. Let $p(x,y)$ be a polynomial, in two non-commuting
    variables $x$ and $y$, with real coefficients.
    If the inequality
    \[ \Tr \left(p(x,y)\right)\leq 0 \]
    is satisfied for all $x,y\in \pA$, then it is also satisfied for
    all $x,y\in \rho (\pA)$. The same assertion remains valid if the
    inequality sign is replaced by equality.
\end{proposition}

\begin{proof}
    If $\pR$ is the radical of $\pA$, then
    $\pA/\pR\cong\pA_1\times\dots\times\pA_s$, where each $\pA_i$ is a
    simple algebra. Let $W_k$ be the unique (up to isomorphism)
    irreducible quaternionic module of $\pA_k$. Then $W_1,\dots,W_s$
    are representatives of the isomorphism classes of irreducible
    quaternionic $\pA$-modules and let $\rho_1,\dots,\rho_s$ be
    their corresponding representations. Let
    $0=V_0\subset V_1\subset\dots\subset V_m=\bH^n$ be a
    Jordan--H\"older series of $\bH^n$ (viewed as a quaternionic
    $\pA$-module). Denote by $n_k$ the number of indices
    $i\in\{1,2,\dots,m\}$ such that $V_i/V_{i-1}\cong W_k.$ As $\bH^n$
    is a faithful $\pA$-module, we have $n_k\geq 1$ for each $k$.
    For any $x,y\in\pA$, we have
    \[ \Tr (p(x,y))=\sum_{i=1}^sn_i\Tr (p(\rho_i(x),\rho_i(y))). \]
    By the above remark, for a fixed $k\in\{1,2,\dots,s\}$ and any $x_k,y_k\in\rho_k(\pA)$
    there exist $x,y\in\pA$ such that $\rho_k(x)=x_k$, $\rho_k(y)=y_k$,
    while $\rho_l(x)=\rho_l(y)=0$ for $l\neq k.$ For such $x,y$ we have
    \[ \Tr (p(x,y))=n_k\Tr (p(x_k,y_k)). \]
    As $n_k \geq 1$ and $\Tr (p(x,y)) \leq 0$, we conclude that
    $\Tr (p(x_k,y_k))\leq 0.$ Since the representation $\rho$
    is equivalent to some $\rho_k$, The first assertion is proved.

    The second assertion is a consequence of the first.
\end{proof}

 We shall also need the following easy lemma.

\begin{lemma}\label{rD}
    Let $\pA=M_r(D)$ where $D\in\{\bR,\bC,\bH\}$.
    \begin{enumerate}
        \item If $\Tr([A,B]^2)\leq 0$ holds for all $A,B\in\pA$ then $r=1$.
        \item $\Tr([A,B]^3)=0$ holds true in $\pA$ if and only if either
            $r=1$ or $r=2$ and $D\in\{\bR,\bC\}.$
    \end{enumerate}
\end{lemma}

\begin{proof}
To prove (1), suppose $r=2$ then the matrix pair
$$
    A=\begin{bmatrix}1&0\\0&0\end{bmatrix},\quad
    B=\begin{bmatrix}0&1\\-1&0\end{bmatrix}
$$
has $\Tr([A,B]^2)=4 > 0$. If $r\geq 3$, we may extend the matrices
$A,B$ with rows and columns of zeros to see that inequality does not
hold. Hence, we must have that $r=1$.

Next we prove (2). If $r=1$ or $r=2$ and $D\in\{\bR,\bC\}$ then
Corollaries \ref{qident} and \ref{disc2} give $\Tr ([A,B]^3)=0$ on
$\pA$.

To prove the converse, we proceed by contradiction. If $r=2$ and
$D=\bH$ then
\[ A=\begin{bmatrix}0&1\\0&\mathsf{j}\end{bmatrix}, \quad
    B=\begin{bmatrix}0&1\\ \mathsf{i}&0\end{bmatrix}\]
satisfy $\Tr([A,B]^3)=-4\neq 0$. Similarly, for $r\geq 3$ we see our
equality is invalid even for $D=\bR:$ Observe that
\[ A=\begin{bmatrix}0&0&0\\0&1&0\\1&0&0\end{bmatrix}, \quad
    B=\begin{bmatrix}1&1&0\\0&0&1\\0&0&0\end{bmatrix}\]
have $\Tr([A,B]^3)=-6\neq0$. Thus, by our argument in part (1) we are
done.
\end{proof}

The following theorem is a quaternionic version of Theorem 1.5.4
from \cite{RR}. The only change in the wording appears in part (4)
where we have replaced their equality with an inequality.

\begin{theorem}\label{nilpotent}
    For a unital subalgebra $\pA\subseteq M_n$, the following are
    equivalent:
    \begin{enumerate}
    \item $\pA$ is triangularizable.
    \item If $A,B\in\pA$ are nilpotent then so is $A+B$.
    \item If at least one of $A,B\in\pA$ is nilpotent then so is $AB$.
    \item $\Tr ([A,B]^2) \leq 0$ for all $A,B\in\pA$.
    \end{enumerate}
\end{theorem}

\begin{proof}
First, if we assume (1) holds then it is trivial to see that
(2),(3), and (4) are all satisfied.

Conversely, suppose at least one of (2),(3) or (4) holds. As in the
proof of Proposition \ref{subalgebra}, choose a Jordan--H\"older
series $0=V_1\subset V_2\subset\cdots\subset V_m=\bH^n$ for the
quaternionic $\pA$-module $\bH^n$. Let $n_k$ be the quaternionic
dimension of the quotient $W_k=V_k/V_{k-1}.$ It suffices to show
that each $n_k=1.$

There is a $Q\in \GL_n(\bH)$ such that $Q\pA Q^{-1}$ consists of
block upper triangular matrices with successive diagonal blocks
square of size $n_k$, $k=1,\dots,m.$ For any $X\in\pA$ let
$\rho_k(X)$ denote the $k$-th diagonal block of size $n_k$ of the
matrix $QXQ^{-1}.$ Each $\rho_k(\pA)$ is a simple real algebra
isomorphic to $M_{n_k}(D_k)$ with $D_k\in\{\bR,\bC,\bH\}.$ By
Noether--Skolem Theorem we may assume that $Q$ is chosen so that
each $\rho_k(\pA)=M_{n_k}(D_k).$

If (2) or (3) holds then $n_k=1$ because otherwise the pair of
nilpotent matrices
$\left[\begin{smallmatrix}0&1\\0&0\end{smallmatrix}\right],
\left[\begin{smallmatrix}0&0\\1&0\end{smallmatrix}\right]$ will add
and multiply to a matrix which is not nilpotent. If (4) holds, then
Lemma \ref{rD} gives that $n_k=1$.
\end{proof}

\begin{remark}
    The equivalence of (1),(2) and (3) was also mentioned by
    Kermani at the recent ILAS Conference \cite{K}.
\end{remark}

Our next objective is to characterize subalgebras of $M_n$ that
satisfy the identity $\Tr ([X,Y]^3)=0$. For that purpose we need the
concept of quasi-triangularizability which we now define

Let us define a $\{1,2\}$\emph{-sequence} as a sequence
$\sigma=(\sigma_1,\sigma_2,\dots,\sigma_m)$ of integers from
\{1,2\}. We say that $m$ is its \emph{length} and
$|\sigma|:=\sigma_1+\dots+\sigma_m$ is its \emph{size}. Assuming
that $|\sigma|=n,$ we denote by $M_\sigma$ the subalgebra of $M_n$
consisting of all block triangular matrices
$$
A=\begin{bmatrix}A_{11} & A_{12} & \cdots & A_{1,m-1} & A_{1m}\\
                 0      & A_{22} &        & A_{2,m-1} & A_{2m}\\
                 \vdots &        &        &           &       \\
                 0      & 0      &        & 0         & A_{mm}
\end{bmatrix}
$$
with the diagonal blocks $A_{ii}\in M_{\sigma_i}$ subject to the
additional condition that $A_{ii}\in M_2(\bC)$ whenever
$\sigma_i=2$, and all $A_{i,j},i<j$, arbitrary quaternionic matrices
of appropriate sizes.

We can now define the quasi-triangularizable sets of quaternionic
matrices.

\begin{definition}\label{quasitri}
    A collection $\pS$ of $n\times n$ quaternionic matrices is
    \emph{quasi-triangularizable} (denoted by q.t.) if
    $P\pS P^{-1}\subseteq M_{\sigma}$ for some $P\in\GL_n(\bH)$
    and some $\{1,2\}$-sequence $\sigma$ of size $n$. (If all
    $\sigma_i$ can be taken to be $1$, then $\pS$ is triangularizable.)
\end{definition}

\begin{theorem}\label{qt}
A unital subalgebra $\pA\subseteq M_n$ is q.t. if and only if
$\Tr([A,B]^3)=0$ for all $A,B\in \pA$.
\end{theorem}

\begin{proof}
If $\pA\subseteq M_n$ is q.t., then for $A,B\in\pA$ there is a
$P\in\GL_n(\bH)$ and a \{1,2\}-sequence $\sigma$ with $|\sigma|=n$
such that $P\pA P^{-1}\subseteq M_{\sigma}$. For $A,B\in \pA$ with
diagonal blocks, denoted $A_1,\dots,A_k$ and $B_1,\dots,B_k$ for
$PAP^{-1},PBP^{-1}$ respectively, we see that all satisfy the trace
identity proven in Corollary \ref{disc2}. In particular the identity
holds for integers $(m,n)=(2,1)$. By Lemma \ref{rD} and the identity
(\ref{jed}), we have
\[ \Tr([A,B]^3)=\sum_{i=1}^k\Tr([A_i,B_i]^3)=0. \]

Conversely, suppose that any $A,B\in\pA$ satisfy the given identity.
If $\pR$ is the radical of $\pA$ then we know that
$\pA/\pR=\pA_1\times\dots\times\pA_s$ where the $\pA_i$'s are simple
$\bR$-algebras. That is, for each $i\in\{1,\dots,s\}$ we have that
$\pA_i\cong M_r(D_i)$, for some positive integer $r$ and
$D_i\in\{\bR,\bC,\bH\}.$ Proposition \ref{subalgebra} guarantees
that our trace identity remains true on each $\pA_i$. Thus, Lemma
\ref{rD} implies that the possibilities for $r_i$ and $D_i$ can be
reduced to exactly one of $r_i=1$ and $D_i\in\{\bR,\bC,\bH\}$ or
$r_i=2$ and $D_i\in\{\bR,\bC\}$.

Next, as in the proof of Proposition \ref{subalgebra}, we choose a
Jordan--H\"older series $0=V_0\subset V_1\subset \ldots\subset
V_m=\bH^n$ for the quaternionic $\pA$-module $\bH^n$.
{}From the fact just proved above it follows that each irreducible
quotient $V_i/V_{i-1}$ has quaternionic dimension $\sigma_i=1$ or 2.
Consequently, there exists $Q\in\GL_n(\bH)$ such that $Q\pA Q^{-1}$
is contained in the subalgebra of $M_n$ consisting of the block
upper triangular matrices whose successive diagonal blocks have
sizes given by the \{1,2\}-sequence
$\sigma=(\sigma_1,\ldots,\sigma_m)$ of size $n$. For any $X\in\pA$
let $\rho_i(X)$ denote the $i$-th diagonal block of the matrix
$QXQ^{-1}$.

Assume that $\sigma_i=2$. Then $\rho_i(\pA)$ is a unital subalgebra
of $M_2$ isomorphic to $M_2(\bR)$ or $M_2(\bC)$. By Noether--Skolem
theorem see \cite[p.230]{RP}, there exists a matrix
$P_i\in\GL_2(\bH)$ such that the subalgebra $P_i\rho_i(\pA)P_i^{-1}$
is exactly equal to $M_2(\bR)$ or $M_2(\bC)$, respectively. If
$\sigma_i=1$ we just set $P_i=[1]$. Let $P\in\GL_n(\bH)$ be the
block diagonal matrix with successive diagonal blocks
$P_1,\ldots,P_m$. Then we have $PQ\pA Q^{-1}P^{-1}\subseteq
M_\sigma$. Therefore, $\pA$ is quasi-triangularizable.
\end{proof}

The referee remarks that the theory of polynomial identities 
(see e.g. \cite{DF}) is relevant for the last theorem.
The algebras $\bH$, $M_2(\bR)$
and $M_2(\bC)$ all satisfy the same polynomial identities over $\bR$
because they are all central simple of dimension 4 over their
respective centers, which need not be $\bR$. Thus a unital
subalgebra  $\pA$ is quasi-triangularizable if and only if
$\pA/\pR$ satisfies all the polynomial identities of $M_2(\bR)$.

\section{Simultaneously triangularizable matrix pairs}
The pairs of matrices over a field that are Simultaneously
triangularizable have been studied for a long time, see e.g the book
\cite{RR} and its references. Most of the known results deal with
the problem of characterizing such matrix pairs. On the other hand
the set of all such matrix pairs does not have a simple description,
apart from a particular case which will be mentioned in the next
section. In this section we make several observations concerning
this problem for quaternionic matrices.

Let us start with the definition.

\begin{definition}
    We denote by $\pW_n$, the set of all matrix pairs $(A,B)\in
    M_n\times M_n$ such that $A$ and $B$ are
    simultaneously triangularizable.
\end{definition}

The problem of describing $\pW_n$ is apparently very hard (see next
section for the case $n=2$). An easy observation is that this set is
semi-algebraic. Indeed, the group $G_n:=\GL_n(\bH)$ is a real
algebraic group and the map
\begin{equation} \label{map}
     G_n\times \pU_n\times \pU_n\to
        M_n\times M_n
\end{equation}
which sends $(g,x,y)$ to $(gxg^{-1},gyg^{-1})$ is regular. Now
observe that $\pW_n$ is the set theoretic image of this map, and it
is well known that the image of a regular map is a semi-algebraic
set.

We proceed to show that $\pW_n$ is a closed set. We claim that the
image of $\Sp(n)\times \pU_n\times\pU_n$ under the map (\ref{map})
is the whole set $\pW_n.$ Indeed, let $(x,y)\in\pW_n$ and choose
$g\in G_n$ and $a,b\in\pU_n$ such that $x=gag^{-1}$ and
$y=gbg^{-1}.$ Let us write $g=ut$ where $u\in\Sp(n)$ and $t\in
G_n\cap\pU_n.$ Then we have $x=ucu^{-1},\,y=udu^{-1}$ with
$c=tat^{-1}$ and $d=tbt^{-1}$ in $\pU_n.$ This proves our claim.

Since $\Sp(n)$ is a compact group and $\pU_n\times\pU_n$ is a closed
set, we infer that $\pW_n$ is closed in the ordinary (Euclidean)
topology. Apparently, this is not true for the Zariski topology (see
Problem 8.3).

Let $\pP_n$ denote the algebra of real polynomial functions on
$M_n\times M_n$. Denote by $\pP_n'$ the subalgebra of
$G_n$-invariant functions, i.e., functions $f\in\pP_n$ such that
\[ f(gxg^{-1},gyg^{-1})=f(x,y) ; \quad
\forall g\in G_n; \forall x,y\in M_n.\] Similarly, let  $\pP_n''$
denote the subalgebra of $\pP_n''$ consisting of $\Sp(n)$-invariant
functions,

Since $\pW_n$ is $G_n$-invariant, its Zariski closure
$\overline{\pW}_n$ is also $G_n$-invariant. Let
$\pI_n\subseteq\pP_n$ be the ideal consisting of all functions $f\in
\pP_n$ that vanish on $\pW_n$, and set $\pI_n'=\pI_n\cap\pP_n'$ and
$\pI_n''=\pI_n\cap\pP_n''$. By the definition of $\overline{\pW}_n$
we have
\[ \overline{\pW}_n=\{(x,y)\in M_n\times M_n: f(x,y)=0, \forall f\in \pI_n\}.\]
By using the fact that $\Sp(n)$ is a compact group, one can easily
show that also
\[ \overline{\pW}_n=\{(x,y)\in M_n\times M_n: f(x,y)=0,
\forall f\in \pI_n'\}.\]

The algebra $\pP_n$ is bigraded: We assign to the $4n^2$ coordinate
functions of the matrix $x$ the bidegree (1,0), and to the
coordinate functions of $y$ the bidegree (0,1). The subalgebras
$\pP_n'$ and $\pP_n''$ inherit the bigradation from $\pP_n.$ The
ideals $\pI_n$, $\pI_n'$ and $\pI_n''$ are also bigraded.

We shall now exhibit two infinite families of concrete polynomials
that belong to $\pI_n'$. Let us first state an obvious result about
matrices with purely imaginary eigenvalues.

\begin{lemma}\label{imaginary}
    If all eigenvalues of $A\in M_n$ are purely imaginary, then
    $\Tr (A^{2k-1})=0,$ $\Tr (A^{4k-2})\leq 0$ and $\Tr (A^{4k}) \geq 0$ for
    all integers $k\geq 1$.
\end{lemma}

It is clear that, for $(X,Y)\in\pW_n$, all eigenvalues of $[X,Y]$
are purely imaginary. Hence, we obtain as a simple corollary from
above our first family of polynomial equations (and inequalities)
that are satisfied on $\pW_n$.

\begin{cor}\label{imaginary2}
    If $\,(X,Y)\in \pW_n$ then
    \[\Tr ([X,Y]^{2k-1})=0,\quad \Tr ([X,Y]^{4k-2})\leq 0,\quad \Tr ([X,Y]^{4k}) \geq
    0\]
    are valid for all integers $k\geq 1$.
\end{cor}

We can use the results of section 4 to obtain our second family.

\begin{cor}\label{eq3}
    If $(X,Y)\in\pW_n$ then
    \[\Tr(X^kY^kX^mY^m-Y^kX^kY^mX^m)=0\]
    for all integers $k,m\geq1.$
\end{cor}

To get more insight into the structure of the set $\pW_n$, we shall
analyze the generic fibres of the first projection map
$\pi_1:\pW_n\rightarrow M_n$. As any matrix $A\in M_n$ is
triangularizable, $\pi_1$ is surjective. We denote by $\pF_A$ the
fibre of $\pi_1$ over $A$, i.e.,
$$
    \pF_A=\pi_1^{-1}(A)=\{(A,B): (A,B)\in \pW_n\}.
$$

We say that a matrix $A\in M_n$ is \emph{generic} if it has $n$
distinct eigenvalues. The set of all generic matrices is an open
dense subset of $M_n$. We shall now describe the generic fibres of
$\pi_1$, i.e., the fibres $\pF_A$ with $A$ generic.

For convenience, let us identify the symmetric group $S_n$ with the
group of $n\times n$ permutation matrices.

\begin{proposition}\label{nfactorial}
    For generic $A\in M_n$, the fibre $\pF_A$ is the union of
    $n!$ real vector spaces, each of dimension $2n(n+1)$. Any two of
    these spaces intersect in a common vector subspace of dimension
    $\geq 4n$.
\end{proposition}
\begin{proof}
    Let $\la_1,\dots,\la_n$ be the distinct eigenvalues of
    $A$. If $P\in G_n$ then
    \[ \pF_{PAP^{-1}}=\pi_1^{-1}(PAP^{-1})=P\pi_1^{-1}(A)P^{-1}=P\pF_AP^{-1}. \]
    Hence, without any loss of generality, we may assume that $A$ is
    a diagonal matrix $A=\diag(\la_1,\dots,\la_n)$. Then it
    suffices to prove that
    \[ \pF_A=\bigcup_{P\in S_n}\pF_{A,P}\]
    where
    \[ \pF_{A,P}=\{A\}\times P\pU_nP^{-1}.\]

    Let $P\in S_n$. Since $P^{-1}\pF_{A,P}P=\{P^{-1}AP\}\times
    \pU_n\subseteq \pU_n\times \pU_n$, we have $\pF_{A,P}\subseteq
    \pW_n$. It follows that $\pF_{A,P}\subseteq \pF_A$ for all $P\in
    S_n$.

    Conversely, let $(A,B)\in \pF_A$. Choose $Q\in G_n$ such
    that
    \[(QAQ^{-1},QBQ^{-1}) \in \pU_n\times \pU_n.\]
    Since $QAQ^{-1}$ has $n$ distinct eigenvalues
    $\la_1,\dots,\la_n$ and $QAQ^{-1}\in \pU_n$, there
    is an invertible upper triangular matrix $R$ such that
    $RQAQ^{-1}R^{-1}$ is a diagonal matrix with diagonal entries
    $\la_1,\dots,\la_n$ in some order. Hence,
    $RQAQ^{-1}R^{-1}=P^{-1}AP$ for some $P\in S_n$.

    It follows that $S:=PRQ$ commutes with $A$ and so $S$ is a
    diagonal matrix. Now $RQBQ^{-1}R^{-1}\in \pU_n$ implies that
    $B\in S^{-1}P\pU_nP^{-1}S=P\pU_nP^{-1}$, i.e., $(A,B)\in
    \pF_{A,P}$. This concludes the proof of the first assertion.

    The second assertion follows from the assertion that, for each
    $P\in S_n,$ $P\pU_n P^{-1}, P\in S_n,$ contains the space of
    diagonal matrices.
\end{proof}

We show next that $\pW_n$ is the image of a smooth map defined on a
suitable vector bundle. The group $T_n=G_n\cap\pU_n$ acts on
$\pU_n\times\pU_n$ by simultaneous conjugation
$(t,x,y)\mapsto(txt^{-1},tyt^{-1}),$ where $t\in T_n$ and
$x,y\in\pU_n.$ There is also the right action of $T_n$ on $G_n$ by
right multiplication. By using these two actions one can construct a
vector bundle
\[ G_n\times_{T_n}(\pU_n\times\pU_n) \]
with base the homogeneous space $G_n/T_n$ and fibre
$\pU_n\times\pU_n.$ For more details about this construction we
refer the reader to \cite[p.46]{GB}.

The map (\ref{map}) induces a smooth map from the above vector
bundle to $M_n\times M_n.$ Since $\pW_n$ is the image of this
induced map, we have
\[ \dim\pW_n\leq\dim \left(G_n\times_{T_n}(\pU_n\times\pU_n)\right)=4n^2+2n(n+1)=2n(3n+1).\]
We shall see next that the equality sign holds here.

\begin{cor}\label{dimension of Wn}
    $\dim \pW_n = 2n(3n+1).$
\end{cor}
\begin{proof}
    Since the generic matrices form an open submanifold of
    $M_n$ and each generic fibre has dimension $2n(n+1)$, we
    conclude that
    \[ \dim \pW_n \geq 4n^2+2n(n+1) = 2n(3n+1). \]
    Hence, equality holds.
\end{proof}

We conclude that $\overline{\pW}_n$ has codimension $2n(n-1)$ in
$M_n\times M_n$. Consequently, $\pI_n'$ must have at last $2n(n-1)$
generators. In the next section we shall see that this bound is too
low when $n=2$.

\section{Matrix pairs in $M_2$ with a common eigenvector}

In this section we shall consider the special case $n=2$. The set
$\pW_2$ can be described also  as the set of all ordered pairs
$A,B\in M_2$ such that $A$ and $B$ share a common eigenvector. For
complex matrices, this special case has been fully resolved 
(see e.g. \cite{TL,RG,F}). Let us recall the result.

\begin{theorem} \label{Friedland}
For a pair of matrices $A,B\in M_2(\bC)$ the following are
equivalent:

(a) $A$ and $B$ are simultaneously triangularizable,

(b) $[A,B]^2= 0$,

(c) $\tr ([A,B]^2)= 0$,

(d) $\tr(A^2B^2-(AB)^2)=0,$

(e) $(2\tr (A^2)-(\tr (A))^2(2\tr(B^2)-(\tr(B))^2)= $ \hfil\break
\hbox{\qquad\qquad\qquad\qquad\qquad\qquad\qquad 
$(2\tr (AB)-\tr (A)\tr (B))^2$.}
\end{theorem}

Clearly, this result is much stronger than what Proposition
\ref{permutable trace} gives in this case.

We continue with an easy lemma of independent interest.

\begin{lemma} \label{pure}
Both eigenvalues of the matrix $A\in M_2$ are purely imaginary if
and only if
    \begin{enumerate}
    \item $\Tr (A)=\Tr (A^3)=0$,
    \item $\Tr (A^2)\leq0$ and
    \item $2\Tr (A^4)\leq (\Tr (A^2))^2 \leq 4\Tr (A^4)$.
    \end{enumerate}
\end{lemma}

\begin{proof}
    Necessity of (1) and (2) follows directly from Lemma
    \ref{imaginary}. For (3), we may assume $A=\left[\begin{smallmatrix}
    \mathsf{i}\al & *\\ 0 & \mathsf{i}\be\end{smallmatrix}\right]$
    with $\al,\be\geq 0.$ Then $\Tr A^2=-2(\al^2+\be^2)$ and
    $\Tr A^4=2(\al^4+\be^4)$. It is clear from this that (3)
    is satisfied.

    Suppose now that the conditions (1-3) hold and let $\la_1,\la_2$
    be the eigenvalues of $A$. If
    \[ f(z)=z^4-e_1z^3+e_2z^2-e_3z+e_4, \]
    is the characteristic polynomial for $\chi_2(A)$ then
    $e_1,e_2,e_3,e_4$ are elementary symmetric functions of the
    eigenvalues $\la_1,\la_2,\bar{\la}_1,\bar{\la}_2$ of
    $\chi_2(A).$ By (1) and Newton's identities, we have
    \[ e_1=e_3=0,\quad e_2=-\frac{1}{2}\Tr A^2, \quad
        e_4=\frac{1}{8}\left((\Tr A^2)^2 - 2\Tr A^4\right). \]
    So we have that
    \[ f(z)=z^4-\frac{1}{2}(\Tr A^2)z^2+\frac{1}{8}\left((\Tr A^2)^2 - 2\Tr A^4\right). \]
    The inequalities of the lemma show that this quadratic polynomial in
    $z^2$ has two real roots, both $\leq0$. Hence the eigenvalues are
    indeed purely imaginary.
\end{proof}

As in the previous section, we obtain the following corollary.

\begin{cor}
    Let $(A,B)\in\pW_2$ and let $\pA\subseteq M_2$ be the unital
    subalgebra generated by $A$ and $B$. Then for all $X,Y\in \pA$
    we have
    \begin{enumerate}
        \item $\Tr([X,Y]^3)=0$,
        \item $\Tr([X,Y]^2)\leq0$, and
        \item $2\Tr([X,Y]^4) \leq (\Tr([X,Y]^2))^2\leq 4\Tr([X,Y]^4)$.
    \end{enumerate}
\end{cor}
\begin{proof}
    It suffices to observe that $\pA\times\pA\subseteq\pW_2$.
\end{proof}

It is known that the algebra $\pP_2'$ (see the previous section for
the definition) has a minimal set of bihomogeneous generators (MSG)
of cardinality 32 (see \cite{DS,DD}). In the remainder of this section
we shall summarize the results that we obtained while trying to
construct an MSG of the ideal $\pI_2'\subseteq\pP_2'$. In our
computations we used the generators constructed in \cite{DD}.

Let $\pI_2'(k,l)$ denote the subspace of $\pI_2'$ consisting of
homogeneous functions of bidegree $(k,l)$ and let $d_{k,l}$ be its
dimension. Let $\pI_2'(s)$ be the sum of the $\pI_2'(k,s-k)$ for
$k=0,1,\dots,s$ and set $d_s=\dim\pI_2'(s).$ Since $\pW_2$ is
invariant under the switching map $(x,y)\mapsto(y,x),$ we have
$d_{l,k}=d_{k,l}$ for all $k$ and $l$. We have computed the
dimensions $d_{k,l}$ for $k+l\leq 15,$ as seen in Figure 1.

\newpage
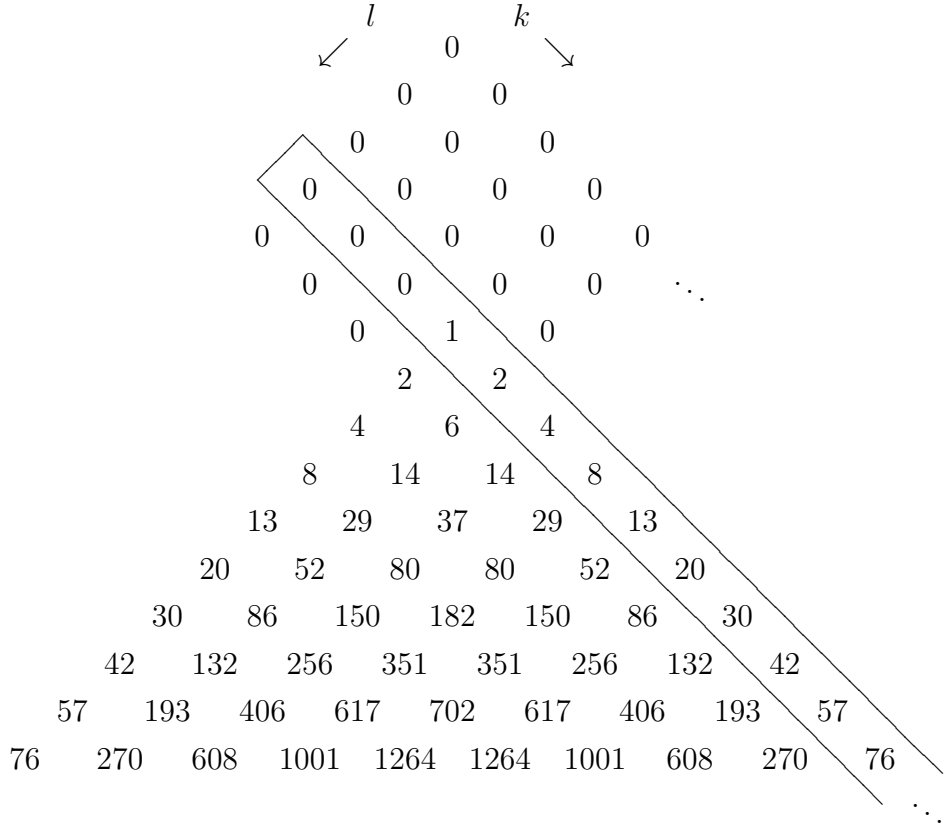
\begin{figure}
\begin{xy}
(111,-21)*{} ; (26,64)**\dir{-} ; (20,58)**\dir{-} ;
(103,-25)**\dir{-} ;(55,80)*{k}; (35,80)*{l}; (60,75)*{\searrow};
(30,75)*{\swarrow}; (-15,0)*{\xymatrix @ur @!=1pt {
           &    &     &     &     &     &    &       & 0  & 0 & 0 & 0 & 0     \\
           &    &     &     &     &     &    &       & 0  & 0 & 0 & 0 & 0     \\
           &    &     &     &     &     &    &       & 0  & 0 & 0 & 0 & 0     \\
        76 & 57 & 42  & 30  &  20 & 13  & 8  &   4   & 2  & 1 & 0 & 0 & 0     \\
           &270 & 193 & 132 &  86 & 52  &29  &  14   & 6  & 2 & 0 & 0 & 0     \\
           &    & 608 & 406 & 256 & 150 &80  &  37   & 14 & 4 &   &   &\ddots \\
           &    &     & 1001& 617 & 351 &182 &  80   & 29 & 8 &   &   &       \\
           &    &     &     & 1264& 702 &351 & 150   & 52 &13 &   &   &       \\
           &    &     &     &     &1264 &617 & 256   & 86 &20 &   &   &       \\
           &    &     &     &     &     &1001& 406   &132 &30 &   &   &       \\
           &    &     &     &     &     &    & 608   &193 &42 &   &   &       \\
           &    &     &     &     &     &    &       &270 &57 &   &   &       \\
           &    &     &     &     &     &    &       &    &76 &   &   &       \\
           &    &     &     &     &     &    &       &    &\ddots &&   &       \\}}
\end{xy}
\caption{
    The dimensions
    $d_{k,l}$; $k,l\geq 0;\,k+l\leq 15.$}
\end{figure}

The sequence, seen isolated in Figure 1,
\[(d_{k,3})_{k\geq0}=0,0,0,1,2,48,13,20,30,42,57,76,\dots\] is
apparently the same as the sequence A061866 in the On-Line
Encyclopedia of Integer Sequences \cite{NS}. The latter sequence
$(a_k)_{k\geq 0}$ has the following definition: The integer $a_k$ is
the number of integer triples $(x,y,z)$ such that $1\leq x<y<z\leq
k$ and $x+y+z\equiv 0\,\pmod{3}.$ The middle ``vertical" sequence
\[ (d_{k,k})_{k\geq0}=0,0,0,1,6,37,180,698,\dots\]
is not recorded in this encyclopedia.

In principle one can construct an MSG of $\pI_2'$ by the following
routine procedure. Denote by $\pJ'_m$ the ideal of $\pP_2'$
generated by the subspaces $\pI_2'(k)$ for $k\leq m.$ Define the
subspaces $\pJ'_m(k,l)$ and $\pJ'_m(s)$ similarly to $\pI_2'(k,l)$
and $\pI_2'(s).$ Clearly we have that
$0=\pJ'_0\subseteq\pJ'_1\subseteq\cdots$ and $\pJ'_m\subseteq\pI_2'$
for all $m$. Since $d_m=0$ for $m<6$, we also have $\pJ'_m=0$ for
$m<6.$ Since $d_6=d_{3,3}=1,\,\pJ'_6$ is generated by a single
polynomial $f_1\in\pJ'_2(3,3),$ see Table 2 and formula (\ref{jed}).
Since $\dim\pJ'_6(3,4)=\dim\pJ'_6(4,3)=1,$ while
$d_{3,4}=d_{4,3}=2,$ the ideal $\pJ'_7$ is generated by $f_1$ and
two new generators: $f_2\in\pI_2'(3,4)$ and $f_3\in\pI_2'(4,3).$
Similarly, $\pJ'_8$ is generated by $f_1,f_2,f_3$ and four new
generators:
\[ f_4\in\pI_2'(3,5);\quad f_5,f_6\in\pI_2'(4,4);\quad f_7\in\pI_2'(5,3).\]

An MSG for $\pJ'_9$ consists of $f_1,\dots, f_7,$ and ten new
generators:
\[ f_8,f_9\in\pI_2'(3,6);\quad f_{10},f_{11},f_{12}\in\pI_2'(4,5);\]
\[f_{13},f_{14},f_{15}\in\pI_2'(5,4);\quad f_{16},f_{17}\in\pI_2'(6,3).\]

To obtain an MSG for $\pJ'_{10},$ one has to add to this MSG of
$\pJ'_9$ additional 19 generators $f_{18},\dots,f_{36}.$ For
$\pJ'_{11}$ we need additional 22 generators.

By Hilbert's Basis Theorem we know that this procedure must
terminate and so $\pJ'_m=\pI_2'$ holds for sufficiently large $m$.
However we do not know the value of $m.$ Our computations suggest
that $\pJ'_{13}=\pI'_2.$

We were able to find the first 92 generators using this procedure
and hence, compute the ideal $\pJ'_m$ for $m\leq 14$. In Table 2, we
give our minimal set of generators for $\pJ'_9.$

\begin{center}
    \textbf{Table 2:} An MSG of the ideal $\pJ'_9$
    $$
    \begin{array}{l c l}
        f_1=\Tr (xy^2x[x,y])              &\quad\quad& f_2=\Tr (xy^3x[x,y])\\
        f_3=\Tr (yx^3y[x,y])              &          & f_4=\Tr (y^2x^2y^2[x,y])\\
        f_5=\Tr (xy^3x[x^2,y])            &          & 
							f_6=\Tr ([x,y]\left[[x^2,y],[x,y^2]\right])\\
        f_7=\Tr (x^2y^2x^2[x,y])          &          & f_8=\Tr (yxy^3xy[x,y]) \\
        f_9=\Tr (\left[[x,y],y\right]^3)  &          & f_{10}=\Tr (y^2x^3y^2[x,y])\\
        f_{11}=\Tr([x,y][x,y^2][x^2,y^2]) &          &
				f_{12}=\Tr (\left[[x,y],x\right]\left[[x,y],y\right]^2)\\
        f_{13}=\Tr (x^2y^3x^2[x,y])       &          & 
								f_{14}=\Tr ([x,y][x^2,y][x^2,y^2])\\
        f_{15}=\Tr (\left[[x,y],y\right]\left[[x,y],x\right]^2) & & 
								f_{16}=\Tr (xyx^3yx[x,y]) \\
        f_{17}=\Tr (\left[[x,y],x\right]^3)
    \end{array}
    $$
\end{center}

By using our MSG for $\pJ'_{11}$, we can show that an MSG for
$\pJ'_{12}$ requires additional 28 generators and $\pJ'_{13}$
requires 6. We find it surprising that an MSG of $\pI_2'$ is so
large (it has at least 92 generators). The number of generators of
the given bidegree (bidegree multiplicity) contained in an MSG of
$\pI'_{2}$ is shown in Figure 2 for all bidegrees $(k,l)$ with
$k+l\leq 14$. The top entry corresponds to the generator $f_1$ of
bidegree (3,3)
\newpage
\begin{figure}\label{fig2}
\begin{xy}
(80,55)*{l}; (94,55)*{k}; (99,50)*{\searrow};
(75,50)*{\swarrow};(143,50)*{6};(143,44)*{7};(143,38)*{8};
(143,32)*{9};(143,26)*{10};(143,19)*{11};(143,13)*{12};
(143,7)*{13};(143,0)*{14};
 (18,0)*{\xymatrix @ur
@!=1pt {
       0 & 0& 1 & 1 & 1 & 2 & 1 & 1 & 1 \\
         & 0& 1 & 4 & 4 & 5 & 3 & 2 & 1 \\
         &  & 0 & 1 & 5 & 6 & 7 & 3 & 1 \\
         &  &   & 0 & 1 & 8 & 6 & 5 & 2 \\
         &  &   &   & 0 & 1 & 5 & 4 & 1 \\
         &  &   &   &   & 0 & 1 & 4 & 1 \\
         &  &   &   &   &   & 0 & 1 & 1 \\
         &  &   &   &   &   &   & 0 & 0 \\
         &  &   &   &   &   &   &   & 0 \\}}
\end{xy}
\caption{
    Bidegree multiplicities of an MSG of $\pI'_{2}.$}
\end{figure}

Let us now describe one of the methods that we used to compute these
generators, along with an example. We begin with the trace functions
that we know vanish on $\pW_2$ (see Corollaries \ref{imaginary2},
and \ref{eq3} for available families). Notice that all these traces
have bidegree of the form $(k,k)$ and thus do not provide us with
all the generators. We make the simple observation that $\pW_2$ is
invariant under the substitution $(x,y)\mapsto(x+\al,y+\be)$ where
$\al,\be\in\bR$.

Consider the partial derivation operators $\Dx$ and $\Dy$ on the
polynomial algebra in two noncommuting indeterminates $x$ and $y$.
For instance
\[ \Dx (xyxy^2) = yxy^2 + xy^3.\]
For a given noncommutative polynomial $p(x,y)$, with $\Tr (p(x,y))$ in
$\pI_2'(k,l)$, we obtain that
\[ 0=\Tr (p(x+\al,y+\be))= \Tr
    \left(\sum_{i,j=0}^{k,l}p_{i,j}(x,y)\al^{k-i}\be^{l-j}\right)\]
which implies that $\Tr (p_{i,j}(x,y))\in\pI_2'(i,j)$ for all $i,j$.

We claim that if $\Tr (p(x,y))\in \pI_2'$ then also
\[ \Tr \left(\Dx p(x,y)\right),\, \Tr \left(\Dy p(x,y)\right)\in \pI_2'.\]
This follows from our observation above, along with the fact that
$\Dx p(x,y)$ is equal to the coefficient of $\al$ in the expansion
of $p(x+\al,y),$ and similarily for the other derivitave.

With this, we give explicit computation of $f_{13}.$
\begin{example}
Consider $p(x,y)=x^3y^3x^2y^2-y^3x^3y^2x^2$. By Corollary
\ref{qident} we have that $\Tr (p(x,y)) \in\pI_2'(5,5).$ Then, we can
find an element from $\pI_2'(5,4)$ by computing
\[\Dy p(x,y)=3x^3y^2x^2y^2+2x^3y^3x^2y-3y^2x^3y^2x^2-2y^3x^3yx^2,\]
and observing that the trace of this element is equal to
\[ -2\Tr \left(x^2y^3x^2[x,y]\right). \]
Thus, we obtain the generator $f_{13}\in\pI_2'(5,4)$.
\end{example}

We have verified using Maple that the Jacobian matrix of the
generators $f_1,f_2,f_3,f_6$ generically has rank 4. This shows that
these generators are algebraically independent and agrees with the
fact that $\pW_2$ has codimension 4.

\section{Some open problems}

We conclude with the list of four open problems related to the
topics discussed in this paper. The first problem, suggested by
Lemma \ref{imaginary}, is about complex numbers.

\begin{problem}
    Let $\la_1,\dots,\la_n\in\bC$ and set
    $\tau_k:=\Re(\la_1^k+\cdots+\la_n^k)$ for $k=1,2,3,\dots$.
    Characterize the sequences $(\la_1,\dots,\la_n)$ for
    which $\tau_{2k-1}=0,\,\tau_{4k-2}\leq 0$ and $\tau_{4k}\geq0$ for
    all integers $k\geq1.$
\end{problem}

We warn the reader that the conditions imposed on the $\tau_k$'s do
not imply that all $\la_i$'s are purely imaginary. Replacing some of
the numbers $\la_i$ with $\bar{\la}_i$ does not affect the
conditions of the problem. Hence, without any loss of generality one
may assume that all $\Im(\la_i)\geq0$.

The problem we discussed in sections 6 and 7 remains open.

\begin{problem}
    Find a finite set of polynomial equations and inequalities
    that define $\pW_2$ as a semi-algebraic set.
\end{problem}

\begin{problem}
    Describe the Zariski closure $\overline{\pW}_2$ and compute an
    MSG for the ideals $\pI_2$, $\pI_2'$ and $\pI_2''$.
    In particular, is it true that the pair
    $\left[\begin{smallmatrix}1&0\\0&0\end{smallmatrix}\right],
    \left[\begin{smallmatrix}0&1\\1&0\end{smallmatrix}\right]$
    belongs to $\overline{\pW}_2?$
\end{problem}

Note that this pair does not belong to $\pW_2.$ We have verified
that all 92 generators of $\pJ'_{14}$ vanish on it.

Finally, Figure 1 suggests the following problem.

\begin{problem}

(a) Prove that the sequence $(d_{k,3})_{k\geq0}$ is identical to the
    sequence A061866. Also construct a bijection from the set of integer triples
    $(x,y,z),$ used in the definition of A061866, to a suitable basis
    of $\pI_2'(k,3).$

(b) Identify the sequence $(d_{k,k})_{k\geq0}.$ For
        instance, find the generating function or an explicit formula for
        the $d_{k,k}$.
\end{problem}

\end{document}